\documentclass[twoside,10pt]{article}
\usepackage{times,amsmath,amssymb,amsthm,cite}
\numberwithin{equation}{section}
\numberwithin{figure}{section}
\newtheorem{theorem}{Theorem}[section]
\newtheorem{lemma}[theorem]{Lemma}

\theoremstyle{definition}

\theoremstyle{remark}

\everymath{\displaystyle}
\date{}
\begin{document}

{\Large {\bf On the Boundedness of Solutions of a Rational System
with a Variable Coefficient }

\vspace{.3in}

\vspace{.2in}


\noindent E. CAMOUZIS

\vspace{.3in}

\noindent {\it American College of Greece, Deree College, 6
Gravias Street, Aghia Paraskevi, 15342 Athens, Greece }

\vspace{.2in}

\begin{abstract}
\noindent We establish the boundedness character of solutions of a
system of rational difference equations with a variable
coefficient.

\end{abstract}

\section{Introduction}

\vspace{.2in} Consider the system of difference equations
\begin{equation}\label{general}
x_{n+1}=\frac{x_n}{y_n} \;\; \text{and} \;\; y_{n+1}=x_n+\gamma_n
y_n, \;\; n=0,1,\ldots
\end{equation}
where $\{\gamma_n\}_{n=0}^{\infty}$ is an arbitrary sequence of
positive real numbers and the initial conditions $x_0$ and $y_0$
are positive real numbers.

\noindent When $\gamma_n=\gamma>1$, the solution
$\{x_n,y_n\}_{n=0}^{\infty}$ converges to $(0,\infty)$ and so it
is unbounded. When $\gamma=1$, the solution
$\{x_n,y_n\}_{n=0}^{\infty}$ satisfies the identity
\[
x_n+y_n+\frac{x_n}{y_n}+\frac{1}{y_n}=x_0+y_0+\frac{x_0}{y_0}+\frac{1}{y_0}=A>2
\]
and it is easy to see that it converges to
\[
(0,\frac{A+\sqrt{A^{2}-4}}{2})
\]
and so is bounded. Finally, when $0<\gamma<1$, it was established
in \cite{CAM} that both components of every solution
$\{x_n,y_n\}_{n=0}^{\infty}$ are bounded from above by a positive
constant. The proof that was presented in \cite{CAM} was based on
the properties of the double sequence of finite sums
\[
\phi(i,n)=\sum_{k=0}^{n} \gamma^{k} x_{k+i+1}, \;\;
i=0,1,\ldots,\;\; n=0,1,\ldots,
\]
for which, as it was shown in \cite{CAM}, it holds that
\[
\lim_{n \rightarrow \infty} \phi(i,n)=\frac{\gamma+x_i}{y_i}, \;\;
i=0,1,\ldots\;.
\]
In this paper we extend the ideas of the proof presented in
\cite{CAM} to establish that when $\{\gamma_n\}_{n=0}^{\infty}$ is
bounded from below and from above by two positive constants
$\gamma^{\prime}$ and $\gamma$, and more precisely,
\[
0<\gamma^{\prime}\leq \gamma_n\leq \gamma<1,
\]
both components of every solution of System (\ref{general}) are
bounded from above by a positive constant. It was also shown in
\cite{CAM} that when $\gamma_n=\gamma \in (0,1)$ and the initial
conditions are positive real numbers, the dynamics of System
(\ref{general}), in terms of boundedness, are equivalent with the
dynamics of the system
\begin{equation}\label{w}
x_{n+1}=\frac{x_n y_n}{x_n+\gamma}\;\; \text{and} \;\;
y_{n+1}=\frac{y_n}{x_n+\gamma}, \;\; n=0,1,\ldots\;.
\end{equation}
More precisely, as it was shown in \cite{CAM}, given a solution
$\{x_n,y_n\}_{n=0}^{\infty}$  of System (\ref{general}) with
$\gamma_n=\gamma>0$, the sequence $\{x_n,w_n\}_{n=0}^{\infty}$,
for which,
\[
w_n=\frac{\gamma+x_n}{y_n}, \;\; n=0,1,\ldots,
\]
satisfies
\begin{equation}\label{w}
x_{n+1}=\frac{x_n w_n}{x_n+\gamma}\;\; \text{and} \;\;
w_{n+1}=\frac{w_n}{x_n+\gamma}, \;\; n=0,1,\ldots\;.
\end{equation}
This  is also true for System (\ref{general}) with the variable
coefficient $\gamma_n$. That is, given a solution
$\{x_n,y_n\}_{n=0}^{\infty}$ of System (\ref{general}), the
sequence  $\{x_n,w_n\}_{n=0}^{\infty}$, where
\[
w_n=\frac{\gamma_{n-1}+x_n}{y_n}, \;\; n=0,2,\ldots,
\]
with $\gamma_{-1}=\gamma_0$, satisfies the system
\begin{equation}\label{wgeneral}
x_{n+1}=\frac{x_n w_n}{x_n+\gamma_{n-1}}\;\; \text{and} \;\;
w_{n+1}=\frac{w_n}{x_n+\gamma_{n-1}}, \;\; n=0,1,\ldots\;.
\end{equation}
Furthermore,
\begin{equation}\label{wgeneral2}
w_{n+1}=\frac{1}{y_n}, \;\; \text{for all} \;\; n \geq 0.
\end{equation}

The following definitions and theorems for double sequences will
be useful in the sequel. Assume that
$\{\phi(k,n)\}_{k,n=1}^{\infty}$, is a double sequence of positive
real numbers.
 Then we say that $\phi(k,n)$ converges to $L \in [0,\infty)$,
if for every $\epsilon>0$, there exists $N(\epsilon)$ such that
\[
\vert \phi(k,n)-L\vert < \epsilon, \;\; \text{for all} \;\; k,n
\geq N.
\]
We write
\[
\lim_{k,n \rightarrow \infty} \phi(k,n)=L,
\]
and $L$ is called the double limit of the sequence. The two limits
\[
\lim_{k \rightarrow \infty} \lim_{n \rightarrow \infty} \phi(k,n)
\;\; \text{and} \;\; \lim_{n \rightarrow \infty} \lim_{k
\rightarrow \infty} \phi(k,n)
\]
are called iterated limits.

\noindent Assume that $\{\phi(k,n)\}$ is a double sequence of
positive real numbers and
\[
(k_1,n_1)<(k_2,n_2)<\ldots<(k_s,n_s)<\ldots
\]
is a strictly increasing sequence of pairs of positive integers.
Then $\{\phi(k_s,n_t)\}$ is a double subsequence of
$\{\phi(k,n)\}$.

The following three theorems will be useful in the sequel. For the
proof see \cite{Habil}.

\begin{theorem} Assume that $\{\phi(k,n)\}_{k,n=1}^{\infty}$ is a double sequence of positive real numbers which is bounded from above by a positive constant.
Also, assume that for each $k \geq 1$
\[
\lim_{n \rightarrow \infty} \phi(k,n)=w_k \;\; \text{exists}.
\]
Then for any subsequence $\{\phi(k_s,n_t)\}$ of $\{\phi(k,n)\}$,
\[
\lim_{t \rightarrow \infty} \phi(k_s,n_t)=w_{k_s} \;\;
\text{exists}\;\; \text{for all} \;\; s.
\]
Furthermore, if
\[
\lim_{k \rightarrow \infty} \lim_{n \rightarrow \infty}
\phi(k,n)=L \;\; \text{exists},
\]
then for any subsequence $\{\phi(k_s,n_t)\}$ of $\{\phi(k,n)\}$,
\[
\lim_{s \rightarrow \infty} \lim_{t \rightarrow \infty}
\phi(k_s,n_t)=L.
\]

\end{theorem}

\begin{theorem} Assume that $\{\phi(k,n)\}_{k,n=1}^{\infty}$ is a double sequence of positive real numbers, which is
bounded from above by a positive constant. Also, assume that
$\{\phi(k_s,n_t)\}$ is a double subsequence of $\{\phi(k,n)\}$
which strictly decreases (resp. increases) to a nonnegative value
$L$ and also
\[
\phi(k_s,n_t)<\phi(i,j), \; (\text{resp.}\phi(k_s,n_t)>\phi(i,j))
\; \text{for all} \;\; (i,j)<(k_s,n_t)
\]
and for all $(k_s,n_t)$. Then
\[
\lim_{s,t \rightarrow \infty} \phi(k_s,n_t)=\lim_{s \rightarrow
\infty} \lim_{t \rightarrow \infty} \phi(k_s,n_t)= \lim_{t
\rightarrow \infty} \lim_{s \rightarrow \infty} \phi(k_s,n_t)=L
\in [0,\infty).
\]

\end{theorem}

\begin{theorem} Assume that $\{\phi(k,n)\}_{k,n=1}^{\infty}$ is a double sequence of positive real numbers such
that
\[
\lim_{n \rightarrow \infty} \phi(k,n) \;\; \text{exists uniformly}
\;\; \text{in} \;\; k
\]
and that
\[
\lim_{k \rightarrow \infty}  \lim_{n \rightarrow \infty}
\phi(k,n)=L.
\]
Then the double limit of the sequence $\{\phi(k,n)\}$ exists and
\[
\lim_{k,n \rightarrow \infty} \phi(k,n)=L.
\]
\end{theorem}

\section{Boundedness}

In this section we establish that both components of every
solution of System (\ref{general}) are bounded from above by a
positive constant.

\begin{theorem} Let $\{x_n,y_n\}_{n=0}^{\infty}$ be a solution of
System (\ref{general}) with positive initial conditions $x_0$ and
$y_0$ and such that
\[
0<\gamma^{\prime}\leq \gamma_n\leq \gamma<1, \;\; \text{for all}
\;\; n \geq 0
\]
and $\gamma^{\prime},\gamma \in (0,1)$. Then both components of
the solution $\{x_n,y_n\}_{n=0}^{\infty}$ are bounded from above
by a positive constant.
\end{theorem}
\noindent The proof of the theorem will be presented at the end of
this section.

\noindent Set $\gamma_{-1}=\gamma_0$. Consider the double sequence
of finite sums
\[
\phi(i,n)=x_{i+1}+\gamma_{i-1} x_{i+2}+\gamma_{i-1}\gamma_{i}
x_{i+3}+\ldots+\gamma_{i-1}\cdots\gamma_{i+n-3} x_{i+n},
\]
with
\[
i=0,1,\ldots \; \; \text{and} \;\; n=1,2,\ldots,
\]
or equivalently,
\begin{equation}\label{phidef}
\phi(i,n)=\sum_{k=0}^{n-1} \mu(i,k) x_{i+k+1},\;\; i=0,1,\ldots,
\;n=1,2,\ldots,
\end{equation}
where for each $i \geq 0$,
\[
\mu(i,k)=\prod_{j=i-1}^{k+i-3}\gamma_j, \;\;k=2,3,\ldots
\]
and
\[
\mu(i,1)=1.
\]
\noindent The following lemmas will be useful in the sequel.
\begin{lemma} It holds that
\[
\lim_{i \rightarrow \infty} \lim_{k \rightarrow
\infty}\mu(i,k)=\lim_{i,k \rightarrow \infty}\mu(i,k)=0.
\]
\end{lemma}
\begin{proof}In view of Theorem 1.3, it suffices to show that
\[
\lim_{k \rightarrow \infty}\mu(i,k)=0
\]
uniformly for each $i$. Indeed, for a given positive number
$\epsilon$ and $i$ arbitrary but fixed, we choose $k> \frac{\ln
\epsilon}{\ln \gamma}+1$, or equivalently $\gamma^{k-1}<\epsilon$.
Then
\[
\mu(i,k)<\gamma^{k-1}<\epsilon
\]
from which the result follows.
\end{proof}

\begin{lemma} Let $\{x_n,y_n\}_{n=0}^{\infty}$ be a solution of
System (\ref{general}). Then for each $i \geq 0$,
\[
\lim_{n \rightarrow \infty} \frac{y_{i+n}}{\mu(i,n+2)}=\infty.
\]
\end{lemma}
\begin{proof} From the first equation of System (\ref{general}) we
see that
\[
\frac{y_{i+n+1}}{\mu(i,n+3)}=\frac{x_{i+n}}{\mu(i,n+3)}+\frac{y_{i+n}}{\mu(i,n+2)},\;\;n=0,1,\ldots
\]
and so the sequence
$\left\{\frac{y_{i+n}}{\mu(i,n+2)}\right\}_{n=0}^{\infty}$ is
strictly increasing. Now assume for the sake of contradiction that
\[
\lim_{n \rightarrow \infty} \frac{y_{i+n}}{\mu(i,n+2)}=L \in
(0,\infty).
\]
Then, there exists a positive number $\epsilon$ arbitrarily small
and a positive integer $N$ sufficiently large such that
\[
y_{n+i}<(L+\epsilon)\mu(i,n+2), \;\; \text{for all} \;\; n \geq N.
\]
From Lemma 2.2, the sequence $\{\mu(i,n+2)\}_{n=0}^{\infty}$
converges to zero. Thus, the sequence $\{y_{i+n}\}_{n=0}^{\infty}$
goes to zero as well. Furthermore,
\[
\mu(i,n+2)=\prod_{j=i-1}^{n+i-1}\gamma_j\leq\gamma^{n+1}, \;\;
\text{for all} \;\; n \geq 0
\]
implies that
\[
y_{i+n}\leq(L+\epsilon)\gamma^{n+1}, \;\; \text{for all} \;\; n
\geq N
\]
and so
\[
x_{i+n+1}=\frac{x_{i+n}}{y_{i+n}}\geq\frac{1}{(L+\epsilon)\gamma^{n+1}}
\cdot x_{i+n}, \;\; \text{for all} \;\; n \geq N
\]
from which it follows that
\[
\lim_{n \rightarrow \infty} x_{i+n+1}=\infty.
\]
However, from the second equation
\[
y_{i+n+1}>x_{i+n}, \;\; \text{for all} \;\; n \geq 0
\]
which contradicts the fact that the sequence
$\{y_{i+n}\}_{n=0}^{\infty}$ converges to 0.

\end{proof}

\begin{lemma} Let $\{x_n,y_n\}_{n=0}^{\infty}$ be a solution of
(\ref{general}). Then for all $i\geq 0$,
\begin{equation}\label{basic}
\frac{x_{i}+\gamma_{i-1}}{y_i}=w_i=\phi(i,n)+\mu(i,n+1) w_{i+n},
\;\; n=1,2,\ldots\;.
\end{equation}

\end{lemma}
\begin{proof} Let $i \geq 0$ be given. Clearly, in view of
(\ref{wgeneral}),
\[
w_i=x_{i+1}+\gamma_{i-1} w_{i+1}
\]
and so the result is true when $n=1$. Assume that $k >1$ and that
\[
w_i=x_{i+1}+\gamma_{i-1}
x_{i+2}+\ldots+\gamma_{i-1}\cdots\gamma_{i+k-3}
x_{i+k}+\gamma_{i-1}\cdots\gamma_{i+k-2} w_{i+k}
\]
\[
=\phi(i,k)+\mu(i,k+1) w_{k+i}.
\]
Then
\[
w_i=\phi(i,k)+\mu(i,k+1)(x_{i+k+1}+\gamma_{i+k-1} w_{i+k+1})
\]
\[
=\phi(i,k)+\mu(i,k+1)x_{i+k+1}+\mu(i,k+1)\gamma_{i+k-1}w_{i+k+1}
\]
\[
= \phi(i,k+1)+\mu(i,k+2)w_{i+k+1}.
\]
The proof is complete.

\end{proof}

\begin{lemma} Let $\{x_n,y_n\}_{n=0}^{\infty}$ be a solution of
(\ref{general}). Then for all $i\geq 0$,
\begin{equation}\label{basic2}
\lim_{n\rightarrow\infty} \phi(i,n)=\sum_{k=0}^{\infty} \mu(i,k)
x_{i+k+1}=\frac{x_{i}+\gamma_{i-1}}{y_i}=w_i.
\end{equation}
\end{lemma}
\begin{proof} The result follows from (\ref{basic}) together with
the fact, in view of (\ref{wgeneral2}) and Lemma 2.3, that
\[
\lim_{n \rightarrow \infty} \mu(i,n+1) w_{i+n}=\lim_{n
\rightarrow\infty} \frac{\mu(i,n+1)}{ y_{i+n-1}}=0.
\]

\end{proof}

\begin{lemma}\label{ds} Let $\{x_n,y_n\}_{n=0}^{\infty}$ be a solution of (\ref{general}) and assume that for an infinite sequence
of positive integers $\{k_i\}_{i=1}^{\infty}$, $\{x_{k_i}\}$ is a
bounded subsequence of $\{x_n\}$ and
\[
\lim_{i \rightarrow \infty}
\frac{x_{k_i}+\gamma_{k_i-1}}{y_{k_i}}= \lim_{i \rightarrow
\infty} w_{k_i}=M \in (0,\infty).
\]
\noindent Then the following statements are true:

\noindent 1.
\[
\lim_{i \rightarrow \infty} \lim_{n \rightarrow \infty}
\phi(k_i,n)=M.
\]

\noindent 2. For any subsequence $\{\phi(k_{i_s},n_j)\}$ of
$\{\phi(k_i,n)\}$, it holds
\[
\lim_{s \rightarrow \infty} \lim_{j \rightarrow \infty}
\phi(k_{i_s},n_j)=M.
\]

\noindent 3.
\[
\limsup_{i,n\rightarrow \infty} \phi(k_i,n) \leq M.
\]

\noindent 4.
\[
\liminf_{i,n\rightarrow \infty} \phi(k_i,n)>0.
\]

\noindent 5.
\[
\liminf_{i \rightarrow \infty} x_{k_i+1} >0.
\]

\end{lemma}
\begin{proof} 1. The proof follows from Lemma 2.5 and the hypothesis.

\noindent 2. The proof is an immediate consequence of the result
of Part 1 and Theorem 1.1, which is presented in the Introduction.

\noindent 3. The proof is an immediate consequence of the fact
that, for each $i \geq 1$,
\[
\phi(k_i,n) < \sum_{k=0}^{\infty}\mu(k_i,n) x_{k+k_i+1}=w_{k_i},
\;\; \text{for all} \;\;n\geq 1
\]
and the hypothesis that $w_{k_i} \rightarrow M$.

\noindent 4. The proof will be by contradiction. Assume for the
sake of contradiction that there exists
 a decreasing subsequence $\{\phi(k_{i_s},n_j)\}_{s,j=1}^{\infty}$ of
$\{\phi(k_i,n)\}$, for which
\[
\lim_{s,j \rightarrow \infty}\phi(k_{i_s},n_j)=0
\]
and
\[
\phi(k_{i_s}, n_j)<\phi(p,q), \;\; \text{for all} \;\; (p,q)<
(k_{i_s}, n_j).
\]
We claim that both $\{k_{i_s}\}$ and $\{n_j\}$ must increase to
infinity. Otherwise, for  $k_{i_s}$ finite and fixed,
\[ \lim_{j \rightarrow
\infty} \phi(k_{i_s},n_j)=0.
\]
In view of the result of Part 2 and the hypothesis, we see that
\[
\lim_{j \rightarrow \infty} \phi(k_{i_s},n_j)=w_{k_{i_s}}>0
\]
which is a contradiction.

\noindent On the other hand  assume that there exists a positive
integer $N$ such that
\[
\lim_{s \rightarrow \infty} \phi(k_{i_s},j)=0, \;\; j=1,2,\ldots,N
\;\; \text{and} \;\; \liminf_{s \rightarrow \infty}
\phi(k_{i_s},N+1)>0.
\]
In view of (\ref{phidef}), as $ s \rightarrow \infty$, it is easy
to see that
\[
x_{k_{i_s}+t}\rightarrow 0, \;\; \text{for all} \;\; t=1,\ldots,N.
\]
By choosing a further subsequence of $\{k_{i_s}\}_{s=1}^{\infty}$,
which for economy in notation we still denote it as $\{k_{i_s}\}$,
it holds that for each $j=-1,0,\ldots,N-2,\;$ the sequence
$\{\gamma_{k_{i_s}+j}\}_{s=1}^{\infty}$ converges to a positive
number. Set
\[
m=\lim_{s \rightarrow \infty} \prod_{j=-1}^{N-2}\gamma_{k_{i_s}+j}
\in (0,\infty).
\]
Clearly, and in view of (\ref{wgeneral}),
\[
w_{k_{i_s}+N} \rightarrow \frac{M}{m}>0.
\]
Therefore,
\[
x_{k_{i_s}+N+1}=\frac{x_{k_{i_s}+N}
w_{k_{i_s}+N}}{\gamma_{k_{i_s}+N-1}+x_{k_{i_s}+N}} \rightarrow 0
\]
and so, in view of (\ref{phidef}),
\[
\lim_{s \rightarrow \infty} \phi(k_{i_s},N+1)=0
\]
which is a contradiction. Therefore, the sequences $\{k_{i_s}\}$
and $\{n_j\}$ are infinite sequences of positive integers and both
increase to infinity. By applying Theorem 1.2, we get
\[
\lim_{s,j \rightarrow \infty}\phi(k_{i_s},n_j)=\lim_{s \rightarrow
\infty} \lim_{j \rightarrow \infty} \phi(k_{i_s},n_j)=0.
\]
On the other hand, by applying the result of Part 2, we see that
\[
\lim_{s \rightarrow \infty} \lim_{j \rightarrow \infty}
\phi(k_{i_s},n_j)=M \in(0,\infty)
\]
which is a contradiction.

\noindent 5. From Part 4, clearly, there exists a positive number
$I$ such that
\[
\phi(k_i,n) >I, \;\; \text{for all} \;\; i,n \geq N.
\]
In particular,
\begin{equation}\label{liminf}
\phi(k_i,N)=x_{k_i+1}+\gamma_{k_i-1}
x_{k_i+2}+\ldots+\gamma_{k_i-1}\cdots\gamma_{k_i-3+N} x_{k_i+N}
>I>0,
\end{equation}
for all $i \geq N$. Now assume for the sake of contradiction and
without loss of generality that
\[
x_{k_i+1} \rightarrow 0.
\]
Note that
\[
w_{k_i}=x_{k_i+1}+\gamma_{k_i-1} w_{k_i+1} \Rightarrow
w_{k_i+1}=\frac{w_{k_i}}{\gamma_{k_i-1}}-\frac{x_{k_i+1}}{\gamma_{k_i-1}}
\]
and so there exists a further subsequence of
$\{k_i\}_{i=1}^{\infty}$, which for economy in notation we still
denote as  $\{k_i\}$, such that
\[
\gamma_{k_i-1}\rightarrow m>0 \;\; \text{and} \;\; w_{k_i+1}
\rightarrow \frac{M}{m},
\]
and so
\[
x_{k_i+2} =\frac{x_{k_i+1} w_{k_i+1}}{\gamma_{k_i}+ x_{k_i+1}}
\rightarrow 0.
\]
By induction, we see that
\[
\lim_{i \rightarrow \infty} x_{k_i+j}=0, \;\; \text{for all}\;\;
j=1,2,\ldots,N.
\]
By taking limits  in (\ref{liminf}), as $i \rightarrow \infty$, we
get   a contradiction.

\end{proof}
\noindent We now present the proof of Theorem 2.1
\begin{proof}Let $\{x_n,y_n\}_{n=0}^{\infty}$ be a solution of System (\ref{general}). First we establish that the
component $\{y_n\}_{n=0}^{\infty}$ of the solution is bounded from
below by a positive constant. Assume for the sake of contradiction
that there exists an infinite sequence of indices
$\{n_i\}_{i=1}^{\infty}$ such that
\[
y_{n_i+1}=x_{n_i}+\gamma_{n_i} y_{n_i}\rightarrow 0.
\]
Clearly,
\[
x_{n_i-t} \rightarrow 0 \;\; \text{and} \;\; y_{n_i-t} \rightarrow
0, \;\; \text{for all} \;\; t=0,1,\ldots\;.
\]
In addition, there exists a sequence of indices
$\{k_i\}_{i=1}^{\infty}$ such that
\[
k_i \leq n_i, \;\; \text{for all} \;\; i,
\]
for which
\begin{equation}\label{y0}
(y_{k_i-1} \geq 1 \;\text{and} \; y_{k_i}<1) \;\text{and} \;
(y_{t}<1, \;\; \text{for all} \;\; t \in \{k_i+1,\ldots,n_i\}),
\end{equation}
because otherwise,
\[
x_{n_i}=\frac{x_0}{\prod_{j=0}^{n_i-1}y_j}>x_0,
\]
which is a contradiction. From
\[
y_{k_i}=x_{k_i-1}+\gamma_{k_i-1} y_{k_i-1} \;\;\text{and} \;\;
y_{k_i-1} \geq 1, \;\; \text{for all} \;\; i,
\]
it follows that
\[
y_{k_i} \geq \gamma_{k_i-1}\geq \gamma^{\prime}, \;\; \text{for
all}\; \; i,
\]
and so
\[
y_{k_i} \in [\gamma^{\prime},1), \;\; \text{for all} \;\; i.
\]
For $i$ sufficiently large, when $r\in \{k_i+1,\ldots,n_i\}$,
\[
x_r =\frac{x_{r-1}}{y_{r-1}}>x_{r-1}
\]
and more precisely,
\[
x_{n_i}>x_{n_i-1}>\ldots>x_{k_i+1}>x_{k_i}.
\]
Therefore,
\[
x_{k_i}<x_{n_i},
\]
from which it follows that $x_{k_i} \rightarrow 0$. By utilizing
the fact that
\[
y_{k_i} \in [\gamma^{\prime},1), \;\; \text{for all} \;\; i,
\]
we may select a further subsequence of $\{k_i\}$, still denoted as
$\{k_i\}$ such that
\[
y_{k_i} \rightarrow L \in [\gamma^{\prime},1] \;\; \text{and} \;\;
\gamma_{k_i-1}\rightarrow l_{-1} \in [\gamma^{\prime},\gamma].
\]
Therefore,
\[
x_{k_i+1}=\frac{x_{k_i}}{y_{k_i}} \rightarrow 0 \;\; \text{and}
\;\;  w_{k_i}=\frac{\gamma_{k_i-1}+x_{k_i}}{y_{k_i}} \rightarrow
\frac{l_{-1}}{L}=M \in \left[\gamma^{\prime},
\frac{\gamma}{\gamma^{\prime}}\right].
\]
By applying Lemma \ref{ds}, we get
\[
\liminf_{i \rightarrow \infty}x_{k_i+1}> 0
\]
which is a contradiction. Hence, the component
$\{y_n\}_{n=0}^{\infty}$ of the solution is bounded from below by
a positive constant $m$. In view of
\[
x_{n+1}=\frac{x_n}{y_n}=\frac{1}{y_{n-1}} \cdot
\frac{x_n}{x_n+\gamma_{n-1}}, \;\; \text{for all} \;\; n \geq 1,
\]
we see that
\[
x_{n+1}<\frac{1}{m}, \;\; \text{for all} \;\; n \geq 1,
\]
and so the component $\{x_n\}_{n=0}^{\infty}$ is bounded from
above. From the second equation of the system, clearly
\[
y_{n+1}<\frac{1}{m}+\gamma y_n, \;\; \text{for all} \;\; n \geq 2,
\]
and so
\[
\limsup_{n \rightarrow \infty} y_n \leq \frac{1}{m(1-\gamma)}.
\]
The proof of the Theorem is complete.

\end{proof}

\end{document}